\documentclass[11pt]{amsart}

\usepackage[all]{xy}
\usepackage{amsmath,amssymb,amsthm,enumerate}

\usepackage[colorlinks, bookmarks=true]{hyperref}

\theoremstyle{plain}
\newtheorem{theorem}{Theorem}[section]
\newtheorem{corollary}[theorem]{Corollary}
\newtheorem{lemma}[theorem]{Lemma}
\newtheorem{proposition}[theorem]{Proposition}

\theoremstyle{definition}

\newtheorem{remark}[theorem]{Remark}
\newtheorem{question}[theorem]{Question}

\newtheorem{case}{Case}

\renewcommand{\leq}{\leqslant}
\renewcommand{\geq}{\geqslant}\usepackage{amssymb}

\newcommand{\vr}{\varepsilon}
\newcommand{\one}{\mathbf{1}}

\newcommand{\be}{\begin{equation}}
\newcommand{\ee}{\end{equation}}

\newcommand{\R}{\mathbb R}

\newcommand{\N}{\mathbb N}

\newcommand{\ave}{{\mathrm{Ave}}}

\newcommand{\ce}{\mathcal E}
\newcommand{\se}{{{\mathfrak{C}}}_{\mathcal{E}}}

\newcommand{\ball}{\mathbf{B}}
\newcommand{\A}{\mathcal{A}}

\newcommand{\cs}{\mathfrak{C}}
\newcommand{\dist}{\mathrm{dist}}
\newcommand{\lore}{\mathfrak{l}}

\def \ker{{\mathrm{ker}} \, }

\renewcommand{\span}{\mathrm{span}}
\def \vr{\varepsilon}

\def \ran{{\mathrm{ran}} \, }

\def \eqalign#1{\null\,\vcenter{\openup\jot 
   \ialign{\strut\hfil$\displaystyle{##}$&$
      \displaystyle{{}##}$\hfil \crcr#1\crcr}}\,}
\def\iK{\mathcal {K}}

\def\iFSS{\mathcal {FSS}}
\def\iSS{\mathcal {SS}}
\def \iIN{\mathcal{IN}}
\def \iWK{\mathcal{WK}}
\def \iDP{\mathcal{DP}}
\def \iSCS{\mathcal{SCS}}

\DeclareMathOperator{\sgn}{sgn}

\begin{document}

\nocite{*}

\numberwithin{equation}{section}

\pagestyle{headings}


\title[Operator ideals]{Operator ideals on non-commutative function spaces}

\author[T. Oikhberg]{Timur Oikhberg}
\address{
Dept.~of Mathematics, University of Illinois at Urbana-Champaign, Urbana IL 61801, USA}
\email{oikhberg@illinois.edu}

\author[E. Spinu]{Eugeniu Spinu}
\address{Dept. of Mathematical and Statistical Sciences, University of Alberta
Edmonton, Alberta  T6G 2G1, CANADA}
\email{espinu@math.ualberta.ca}

\begin{abstract}
Suppose $X$ and $Y$ are Banach spaces, and ${\mathcal{I}}$, ${\mathcal{J}}$
are operator ideals  (for instance, the ideals of strictly singular, weakly compact, or
compact operators). Under what conditions does the inclusion
${\mathcal{I}}(X,Y) \subset {\mathcal{J}}(X,Y)$, or the equality
${\mathcal{I}}(X,Y) = {\mathcal{J}}(X,Y)$, hold?
We examine this question when ${\mathcal{I}}, {\mathcal{J}}$ are the ideals of
Dunford-Pettis, strictly (co)singular,
finitely strictly singular, inessential, or (weakly) compact operators,
while $X$ and $Y$ are non-commutative function spaces.
Since such spaces are ordered, we also address the same questions for positive parts
of such ideals. 
\end{abstract}

\thanks{The authors acknowledge the generous support of Simons Foundation,
via its Travel Grant 210060. They would also like to thank the organizers of
Workshop in Linear Analysis at Texas A\&M, where part of this work was
carried out.}

\maketitle

\section{Introduction}\label{s:intro}

Let $X$ and $Y$ be Banach spaces. We denote by $B(X,Y)$ the set of all linear bounded operators
acting between $X$ and $Y$. In this paper, we address the following questions.
When do certain classical operator ideals in $B(X,Y)$ coincide? What are the necessary
and sufficient conditions for an operator in $B(X,Y)$ to belong to a given ideal?
We concentrate on the case when $X$ or $Y$ are non-commutative function or sequence spaces.
These spaces are ordered, so we also consider the same questions 
for the subsets of positive operators in those ideals.

Here and below, ``operator ideals'' are understood in the sense of \cite{Pie}.
More precisely: the operator ideal ${\mathcal{I}}$ ``assigns'' to any pair of Banach
spaces $X$ and $Y$, a linear subspace ${\mathcal{I}}(X,Y) \subset B(X,Y)$, in such a way
that the following conditions are satisfied:
\begin{enumerate}
\item 
If $T \in B(X,Y)$ has finite rank, then $T \in {\mathcal{I}}(X,Y)$.
\item
If $T \in {\mathcal{I}}(X,Y)$, $U \in B(X_0,X)$, and $V \in B(Y,Y_0)$,
then $VTU \in {\mathcal{I}}(X_0,Y_0)$.
\end{enumerate}

Several ideals of operators will appear throughout this paper. The well-known ideals
of compact, respectively weakly compact, operators, are denoted by $\iK$ and $\iWK$.
An operator $T \in B(X,Y)$ is called \emph{Dunford-Pettis} (DP for short) if it takes weakly
compact sets to relatively compact sets. Equivalently, it maps weakly null sequences
to norm null sequences (see e.g. \cite[Chapter 5]{AK} for more information).
$T$ is said to be \emph{finitely strictly singular} (FSS)  if for every
$\vr>0$ there exists $N \in \mathbb{N}$ such that every subspace of dimension at
least $N $ contains a normalized vector $x$ such that $\|Tx\| <\vr$.
$T$ is \emph{strictly singular} (SS) if it is not an isomorphism on any
infinite-dimensional subspace of $X$. $T$ is \emph{strictly cosingular} (SCS) if,
for any infinite-dimensional subspace $Z \subset Y$, the operator $Q_Z T$ is not
surjective ($Q_Z : Y \to Y/Z$ is the quotient map). And it is \emph{inessential} (IN)
or \emph{ Fredholm Perturbation} if $I+UT$ is Fredholm for every operator $U:Y \to X$.
The set of all DP operators from $X$ to $Y$ is denoted by $\iDP(X,Y)$.
The notations $\iFSS$, $\iSS$, and $\iSCS$ have similar meaning.
The classes $\iDP$, $\iFSS$, $\iSS$, $\iSCS$,
and $\iIN$ are closed operator ideals, in the sense of \cite{DJP}. 
For more information, the reader is referred to \cite{Ai}, \cite{Pie}, \cite{Pli}.

It is known that, for any Banach spaces $X$ and $Y$,
\begin{eqnarray*}
\iK(X,Y) \subseteq \iFSS(X,Y) \subseteq \iSS(X,Y) \subseteq \iIN(X,Y),\\
\iK(X,Y) \subseteq \iSCS(X,Y) \subseteq \iIN(X,Y) .
\end{eqnarray*}
In general, these inclusions are proper, and there are no other inclusions
between these classes. However, for specific classes of Banach spaces,
other inclusions, or identities, of these ideals occur. First results of this type
date back to J.~Calkin \cite{Ca}, who (essentially) proved that, if $X$ and $Y$ are
Hilbert spaces, then all these operator ideals coincide.
This was generalized to the spaces of operators on $\ell_p$ ($1 \leq p < \infty$)
and $c_0$ in \cite{GFM}.
A.~Pelczynski \cite{Pe65-I, Pe65-II} proved that, if $Z = C(K)$ ($L_1(\nu)$),
then $\iWK(Z) = \iSS(Z)$ (resp. $\iWK(Z) = \iSCS(Z)$).
Since then, a range of related results were obtained.
A good overview can be found in \cite{Ai}.
Most recently, in \cite{Lef}, ideals of operators on certain subspaces of $C(K)$
were investigated. It was shown, for instance, that, for operator $T$ on the disc
algebra, the following are equivalent: (i) $T$ is finitely strictly singular;
(ii) $T$ is strictly singular; (iii) $T$ is weakly compact.

In this paper, we explore membership in these ideals for operators acting
on non-commutative $L_p$ spaces and Schatten spaces.

To define a non-commutative $L_p$,
let $\A$ be a von Neumann algebra equipped with a faithful normal semifinite trace
$\tau$, and $\tilde{\A}$ be the algebra of all $\tau$-measurable operators affiliated
with $\A$, see \cite{FK}. Then we can define
$L_p(\tau)=\{A\in \tilde{\A}: \tau(|A|^p)< \infty \}$, where
$|A|=(A^*A)^\frac{1}{2}$. It is a Banach space with respect to the norm
$\|A\|=(\tau(|A|^p))^\frac{1}{p}$. For the convenience of notation,
we write $L_\infty(\tau) = \A$.

In this article we are interested in two cases.
First, when $\A$ is hyperfinite (that is, the weak closure of the union of finite
dimensional von Neumann subalgebras) and $\tau$ is finite. In this case
$L_\infty(\tau) \subsetneq L_p(\tau) \subsetneq L_q(\tau)$,
when $1 \leq q<p< \infty$. Second, when  $\A=B(H)$, for a separable Hilbert space
$H$ and $\tau$ is the canonical trace on $B(H)$. In this case the construction
of $L_p(\tau)$ yields the Schatten space
$\cs_p(H)=\{A \in \iK(H): \sum s_i(A)^p < \infty\}$.
Here and below,  $(s_i(A))$ is the sequence of singular values of $A$. We have
$\cs_p(H) \subsetneq  \cs_q(H) \subsetneq \cs_\infty(H) = \iK(H)$
if $1 \leq p<q< \infty$.

The construction of Schatten spaces described here can be generalized: if
$\ce$ is a symmetric sequence space, we  can define $\se(H)$ as the space of all
$A \in \iK(H)$ for which the sequence $(s_i(A))$ belongs to $\ce$, equipped with the norm
$\|A\|_\ce = \|(s_i(A))_{i \in \N}\|_\ce$ (see \cite{GK} or \cite{Si}
for more details).
Then $\cs_p(H)$ is just a shorthand for $\cs_{\ell_p}(H)$.
When $H$ is separable, we use the notation $\se$ and $\cs_p$ for $\se(H)$
and $\cs_p(H)$, respectively.

The spaces described above are equipped
with a positive cone: $L_p(\tau)_+ = L_p(\tau) \cap \tilde{\A}_+$, and
$\se(H)_+ = \se \cap \iK(H)_+$.
In this setting, the (operator theoretic) adjoint of $x \in L_p(\tau)$
(or $x \in \se$) is well defined, and shall be denoted by $x^\star$
($^\star$ is used for the adjoint in
order to distinguish it from $^*$, signifying duality).

The paper is structured as follows. In Section \ref{s:Lp}, we consider
operators on $L_p(\tau)$, where $\tau$ is a faithful normal trace on a finite
hyperfinite algebra. Generalizing \cite{We}, we show that $T\in B(L_p(\tau))$
is not strictly singular if and only if $L_p(\tau)$ contains a subspace $E$,
isomorphic to either $\ell_2$ or $\ell_p$, so that $T|_E$ is an isomorphism,
and both $E$ and $T(E)$ are complemented (Theorem \ref{t:main}).
A similar result is obtained for operators from $L_{p_1}(\tau_1)$
to $L_{p_2}(\tau_2)$ (Theorem \ref{t:main'}).
We also show that, if $1 \leq p_2 \leq p_1 < \infty$, then
$\iSS(L_{p_1}(\tau_1),L_{p_2}(\tau_2))_+ = \iK(L_{p_1}(\tau_1),L_{p_2}(\tau_2))_+$
(Theorem \ref{t:positive}).


In Section \ref{s:lore} we consider operators on Schatten spaces $\se$, where
$\ce = \lore(w,p)$ is a Lorentz space. Here,
$w=(w(k))$ is a non-increasing sequence of positive numbers, so that
$\sum_k w(k) = \infty$, and $\lore(w,p)$ is 
the completion of $c_{00}$ with respect to the norm
$\|(x_i)\| = \big( \sum_k w(k) x_k^{\dagger p} \big)^{1/p}$, where $(x_k^\dagger)$
is the non-increasing rearrangement of $(x_k)$
($\lore(w,p)$ is isomorphic to $\ell_p$ if and only if $\lim_k w(k) > 0$).
We establish an analogue of Theorem \ref{t:main}: $T \in B(\cs_{\lore(w,p)})$
is not strictly singular if and only if there exists a subspace $E$ of
$\cs_{\lore(w,p)}$, isomorphic to either $\ell_2$ or $\ell_p$, so that
$T|_E$ is an isomorphism, and both $E$ and $T(E)$ are complemented
(Theorem \ref{thm:lor SS=IN}).
We also establish a sufficient condition for the sets
$\iK(\cs_{\lore(w_1,p_1)}, \cs_{\lore(w_2,p_2)})$,
$\iFSS(\cs_{\lore(w_1,p_1)}, \cs_{\lore(w_2,p_2)})$, and
$\iSS(\cs_{\lore(w_1,p_1)}, \cs_{\lore(w_2,p_2)})$ to be different
(Proposition \ref{p:FSS not SS}).

In Section \ref{s:Cp}, we restrict our attention to the spaces $\cs_p$.
Sharpening Proposition \ref{p:FSS not SS},
we show that $\iSS(\cs_p,\cs_q)=\iK(\cs_p,\cs_q)$ if $\infty \ge  p\ge 2\ge q \ge 1$,
and otherwise,
$\iSS(\cs_p,\cs_q) \supsetneq \iFSS(\cs_p,\cs_q) \supsetneq \iK(\cs_p,\cs_q)$
(Theorem \ref{t:C_p->C_q}).
Similar coincidence results are established for positive operators
(Theorems \ref{t:C_p->C_q +} and \ref{t:C_p->C_p +}).
Although the dual of a strictly singular operator need not be strictly singular,
Corollary~\ref{c:duality} shows that $T \in \iSS(Z)$ if and only if $T^{*} \in \iSS(Z^*)$
whenever $Z$ is $\cs_p$ ($1 < p \leq \infty$), or $Z = L_p(\tau)$ ($1 < p <\infty$,
$\tau$ is a finite trace on a hyperfinite von Neumann algebra). We show this
result fails when $Z = \cs_1$.
Finally, we prove that, for a Banach space $Z$, $T \in B(\cs_1, Z)$ is Dunford-Pettis
iff its restriction to every copy of $\ell_2$ is compact (Proposition \ref{p:l_2 DP}).

Finally, in Section \ref{s:prod} (Theorem \ref{t:milman}), we show:
if $Z$ is either $\cs_p(H)$ ($1 \leq p \leq \infty$,
$H$ is a Hilbert space), or $L_p(\tau)$ ($1 < p < \infty$, $\tau$
is a normal faithful finite trace on a hyperfinite von Neumann algebra),
and $T, S \in \iSS(Z)$, then $TS \in \iK(Z)$. 

Note that, in the companion paper \cite{OS_Cstar}, we study the same ideals
of operators acting on $C^*$-algebras. Among other things, we prove an
analogue of Theorem \ref{t:main}: an operator on a $C^*$-algebra is strictly
singular if and only if it preserves either a copy of $c_0$, or a copy of $\ell_2$.
We also use inclusion and coincidence of operator ideals to describe algebraic
properties of $C^*$ and von Neumann algebras. Here is a sample result:
for a von Neumann algebra $\A$, $\iFSS(\A)=\iSS(A)=\iIN(\A)=\iWK(\A)$
if and only if $\A$ is of finite type $I$. Otherwise, these classes are different.

Throughout the paper, we adhere to the standard Banach space terminology.
The word ``subspace'' refers to a closed infinite dimensional subspace,
unless specified otherwise. We denote the unit ball of a normed space $X$ by $\ball(X)$.
The term \emph{$\ell_p$-basis} is used as a shorthand for
``a sequence equivalent to the canonical basis of $\ell_p$.''
We say that an operator $T:X \to Y$ \emph{preserves} (or \emph{fixes})
a Banach space $Z$ if
$X$ contains a subspace $Z$, isomorphic to $E$, so that
$T|_E$ is an isomorphism.

\section{Non-commutative $L_p$: continuous case}\label{s:Lp}

The main goal of this section is to describe strictly singular operators
on non-commutative $L_p$ spaces, and then to specialize to positive operators.
In the commutative setting, such a characterization was obtained in \cite{We},
and then generalized in \cite{Ru, Tra}. To achieve our goal, we need to analyze
sequences in $L_p(\tau)$.

\subsection{Unconditional Subsequence Property}

We say that a Banach space has the {\it Unconditional Sequence Property}
({\it USP}) if every weakly null seminormalized sequence contains an unconditional subsequence.

Bellow we will show that a non-commutative $L_p(\tau)$ corresponding to
a hyperfinite von Neumann algebra with a faithful finite normal trace has the USP.

\begin{proposition}\label{p:FDD}
Suppose $\A$, $\tau$, and $p$ are as in Theorem \ref{t:main}.
Then any separable subspace of $L_p(\tau)$ $(1 \le p <\infty)$ is contained in a subspace
with an unconditional FDD. Consequently, if $\A$ is separably acting,
then $L_p(\tau)$ has an unconditional FDD.
\end{proposition}

\begin{remark}\label{r:separably}
By \cite[Lemma 1.8]{Yama}, for a von Neumann algebra $\A$ with a
normal faithful semifinite trace $\tau$, the following
are equivalent: (i) $\A$ is separably acting, (ii) $\A$ has a separable predual,
and (iii) $L_2(\tau)$ is separable. Consequently, these statements are
equivalent to $L_p(\tau)$ being separable, for any (equivalently, all)
$p \in [1,\infty)$.
\end{remark}

\begin{proof}
If $\A$ is a hyperfinite von Neumann algebra, it contains a net $(\A_\alpha)$
of finite dimensional von Neumann subalgebras, ordered by inclusion, so that
$\A$ is the weak$^*$-closure of $\cup_\alpha \A_\alpha$. The conditional
expectations $Q_\alpha : \A \to \A_\alpha$ are completely contractive, and
satisfy $Q_\alpha Q_\beta = Q_\beta Q_\alpha = Q_\alpha$ whenever
$\alpha \leq \beta$. By \cite[Theorem 3.4]{PiLP}, $Q_\alpha$ extends to a
completely contractive map from $L_p(\tau)$ to $L_p(\tau_\alpha)$, where
$\tau_\alpha$ is the restriction of $\tau$ to $N_\alpha$, and $L_p(\tau)$
is the norm closure of $\cup_\alpha L_p(\tau_\alpha)$.

Now suppose $(x_k)$ is a dense subset of a subspace $X \subset L_p(\tau)$.
Then there exists an increasing sequence $(\alpha_k)$ so that
$\max_{j \leq k} \dist(x_j,\A_{\alpha_k}) < 4^{-k}$ for any $k$.
Now define $\A^\prime$ as the weak$^*$ closure of $\cup_k \A_{\alpha_k}$ in $\A$,
and let $\tau^\prime$ be the restriction of $\tau$ to $\A^\prime$.
As noted in the proof of \cite[Theorem 3.4]{PiLP},
$L_p(\tau^\prime) = \overline{\cup_k L_p(\tau_{\alpha_k})}$, and
this space contains $X$. By \cite[Section 7]{PX} (or \cite{Ran02}),
the subspaces $L_p(\tau_{\alpha_k}) \cap \ker Q_{\alpha_{k-1}}$
form an unconditional FDD.
\end{proof}

Theorem~\ref{p:FDD}, combined with \cite[Theorem~1.g.5]{LT1},
\cite[Proposition~1.a.12]{LT1}, and the fact that every normalized block sequence
of an unconditional basis is unconditional imply the following.

\begin{corollary}\label{c:uncond}
Suppose $\tau$ is a normal faithful semifinite trace on a hyperfinite
von Neumann algebra $\A$. Then, for $1 < p < \infty$, $L_p(\tau)$ has the USP.
\end{corollary}

The USP of commutative $L_p$ spaces ($1 < p < \infty$) is well known, and follows from
the unconditionality of the Haar basis. It was proved in \cite{JMS} that $L_1(0,1)$
fails the USP. In the case of non-commutative $L_1$, we have the following.

\begin{proposition}
Let $\tau$ be a normal faithful semifinite trace on a von-Neumann algebra $\A$. Then $L_1(\tau)$ has USP if and only if $A$ is atomic of type I.  
\end{proposition}

\begin{proof}
If $A$ is not atomic of type I then $L_1(\tau)$ contains a complemented copy of $L_1(0,1)$ by \cite[Theorem~1.5.3]{OS}, and, therefore, it fails the USP.  Otherwise,  $L_1(\tau)$ has the USP since it can be written as $\sum_i (\cs_1(H_i))_{\ell_1}$, where $H_i$ is a Hilbert space.   
\end{proof}

\begin{question}\label{q:USP}
Suppose $\tau$ is a normal faithful semifinite trace on a von Neumann
algebra $\A$, and $1 < p < \infty$. Does $L_p(\tau)$ have the USP?
\end{question}

Note that, in the general case, Theorem~\ref{p:FDD} can no longer be used to
prove that $L_p(\tau)$ has the USP: by \cite[Theorem 2.19]{JNRX},
$L_p(\tau)$ may fail the Approximation Property (hence it cannot have
the unconditional FDD).

\subsection{Characterization of strictly singular operators}

The main result of this subsection is:

\begin{theorem}\label{t:main}
Suppose $1 < p < \infty$, $\tau$ is a faithful normal finite trace on a 
von Neumann algebra $\A$, so that $L_p(\tau)$ has Unconditional Subsequence Property.
For $T \in B(L_p(\tau))$, the following statements
are equivalent: 
\begin{enumerate}
\item
$T$ is not strictly singular.
\item
$L_p(\tau)$ contains a subspace $E$, isomorphic either
to $\ell_p$ or $\ell_2$, so that $T|_E$ is an isomorphism, and both
$E$ and $T(E)$ are complemented.
\end{enumerate}
\end{theorem}

Throughout, we assume $p \neq 2$, and $\tau(\one) = 1$.
The implication $(2) \Rightarrow (1)$ is clear. Proving
$(1) \Rightarrow (2)$ is easy for $2 < p < \infty$, due to
Kadec-Pelczynski dichotomy (see e.g. \cite[Theorem 0.2]{RX}): any
subspace of $L_p(\tau)$ contains a further subspace
$E$, isomorphic to either $\ell_p$ or $\ell_2$, and complemented
in $L_p(\tau)$. In fact, for $2 < p < \infty$ our conclusion remains
true even for any normal faithful semifinite trace $\tau$ on a
von Neumann algebra $\A$ (not necessarily hyperfinite).
Below, we use some ideas from \cite{We} to tackle the case of $1 < p < 2$.

The following lemma can be deduced from Rosenthal's characterization
of $\ell_1$-bases. We present an easy proof for the sake of completeness.

\begin{lemma}\label{l:null}
A seminormalized unconditional sequence in a Banach space is either weakly null,
or contains a subsequence equivalent to $\ell_1$. Consequently, any
 unconditional basic sequence in a reflexive space is weakly null.
\end{lemma}

\begin{proof}
Suppose a normalized sequence $(x_n)$, with an unconditional
constant $C$, is not weakly null.
Passing to a subsequence, we find a norm one $x^* \in X^*$ so that
$|x^*(x_n)| > c > 0$ for every $n$.
For any finite sequence $(\alpha_n)$  let $\omega_n = \sgn (\alpha_n)\frac{|x^*(x_n)|}{x^*(x_n)}$. Then
$$  \eqalign{
\sum_n |\alpha_n| \geq \|\sum_n \alpha_n x_n\|
&
 \geq
 C^{-1} \|\sum_n \alpha_n \omega_n x_n\|
\cr
&
 \ge
 C^{-1} |x^*( \sum_n \alpha_n \omega_n x_n)| \geq
 c C^{-1} \sum_n |\alpha_n| .
}  $$
Thus, $(x_n)$ is equivalent to the $\ell_1$-basis.
\end{proof}

\begin{proposition}\label{p:complem}
If $1 < p < 2$, then any sequence in $L_p(\tau)$, equivalent to the
$\ell_2$-basis, has a subsequence whose linear span is complemented.
\end{proposition}

\begin{proof}
Suppose $(x_n)$ is a sequence, equivalent to the $\ell_2$-basis.
By the Hahn-Banach Theorem, $L_q(\tau)$ (here, as before, $1/p + 1/q = 1$)
contains a bounded seminormalized sequence $(y_n)$. By passing to a subsequence,
we may assume $y_n \to y$ weakly. Note that, for any $n$,
$y(x_n) = \lim_m y_m(x_n) = 0$, hence the sequence $z_n = y_n - y$
is weakly null, and biorthogonal to $(x_n)$. By passing to a further
subsequence, and using the non-commutative Kadec-Pelczynski dichotomy
\cite[Theorem 5.4]{RX}, we assume that $(z_n)$ is equivalent either
to the $\ell_2$-basis, or to the $\ell_q$-basis, and complemented.
The latter, however, is impossible. Indeed, suppose there exists a constant
$C$ so that, for every sequence $(\alpha_i)$,
$C^{-1} (\sum_i |\alpha_i|^2)^{1/2} \leq \|\sum_i \alpha_i x_i\| \leq
C (\sum_i |\alpha_i|^2)^{1/2}$, and 
$C^{-1} (\sum_i |\alpha_i|^q)^{1/q} \leq \|\sum_i \alpha_i z_i\| \leq
C (\sum_i |\alpha_i|^q)^{1/q}$. In particular, for any $m$,
\begin{eqnarray*}
C m^{1/q} &\ge& \|\sum_{i=1}^m z_i\| = \sup_{\|x\|_p \leq 1} |(\sum_{i=1}^m z_i)(x)|\\
                   &\geq& |(\sum_{i=1}^m z_i)(C^{-1} m^{-1/2} \sum_{i=1}^m x_i)| = C^{-1}m^{1/2},
\end{eqnarray*}
which fails for sufficiently large values of $m$.

Thus, $(z_n)$ is equivalent to the $\ell_2$-basis, and there exists a
projection $P$ from $L_q(\tau)$ onto $Z = \span[z_n : n \in \N]$.
Note that the restriction of $P^*$ onto $X = \span[x_n : n \in \N]$ is an isomorphism.
Indeed, for any sequence $(\alpha_n) \in \ell_2$, we have
$\|\sum_n \alpha_n x_n\| \sim (\sum_n |\alpha_n|^2)^{1/2}$.
Furthermore, let $z = (\sum_n |\alpha_n|^2)^{-1/2}
 \sum_n \overline{\alpha_n} z_n$. Then $Pz = z$, and
$\|z\| \lesssim 1$, hence
$$
\|P^* (\sum_n \alpha_n x_n)\| \gtrsim \|(P^* (\sum_n \alpha_n x_n))(z)\| =
\|(\sum_n \alpha_n x_n)(Pz)\| \gtrsim (\sum_n |\alpha_n|^2)^{1/2} .
$$
To complete the proof, note that $U^{-1} P^*$ is a bounded projection onto
$X$ where by $U$ we denoted the restriction of $P^*$ onto $X = \span[x_n : n \in \N]$,
viewed as an operator into $\ran P^*$.
\end{proof}

Suppose $\tau$ is a normal faithful semifinite trace on a von
Neumann algebra $\A$. We say that $K \subset L_p(\tau)$ is
\emph{$p$-equiintegrable} if
$\lim_\alpha \sup_{h \in K} \|e_\alpha h e_\alpha\|_p = 0$
for every net of projections $(e_\alpha)$, converging (weakly) to $0$
(see e.g. \cite[Section II.2]{Tak} for a discussion on various modes
of convergence). By \cite[Section 4]{RX}, the following are equivalent:
\begin{enumerate}
\item
$K$ is $p$-equiintegrable.
\item
$\lim_n \sup_{h \in K} \|e_n h e_n\|_p = 0$
for every sequence of projections $(e_n)$, converging (weakly) to $0$.
\item
$\lim_\alpha \sup_{h \in K} \|x_\alpha h y_\alpha\|_p = 0$
if the nets of positive operators $(x_\alpha)$ and $(y_\alpha)$
converge to $0$ weak$^*$.
\end{enumerate}

The following result seems to be folklore.

\begin{lemma}\label{l:equi}
Suppose $K$ is $p$-equiintegrable. Then for every $\vr > 0$ there
exists $\delta > 0$ such that $\sup_{f \in K} \max\{\|e f\|, \|f e\|\} < \vr$
whenever $e$ is a projection of trace not exceeding $\delta$.
\end{lemma}

\begin{remark}\label{r:equi}
If $\tau$ is finite, then a sequence of projections $(e_n)$ converges
weakly to $0$ if and only if $\lim \tau(e_n) = 0$. In this setting, the above
lemma shows that $K$ is $p$-equiintegrable if and only if it is
(in the terminology of \cite{RX}) $K$ is $p$-biequiintegrable.
If $\tau$ is not finite, $p$-equiintegrability need not imply
$p$-biequiintegrability.
\end{remark}

\begin{proof}
Note that, if $(e_n)$ is a sequence of projections, and $\lim_n \tau(e_n) = 0$,
then $e_n \to 0$ weak$^*$. Indeed, otherwise, by passing to a subsequence,
we can find $x \in L_1(\tau)$ and $c > 0$, so that $|\tau(x e_n)| > c$
for any $n$. By polarization, we can assume that $x \geq 0$. Denote by
$\mu_x(t)$ the generalized singular value function of $x$. Then
(see e.g. \cite{DDdP93}) $\|x\|_1 = \int \mu_x(t) \, dt$, and
$\tau(x e_n) \leq \int_0^{\tau(e_n)} \mu_x(t) \, dt$. The latter
converges to $0$, leading to a contradiction.

Find $\delta$ so that $\sup_{f \in K}\|r f r\| < \vr/4$ whenever
$r$ is a projection with $\tau(r) < 2 \delta$.
If $\tau(e) < \delta$, denote by $e^\prime$ the range  projection
of $f e$. Clearly, $\tau(e^\prime) \leq \tau(e)$. Let $r = e \vee e^\prime$.
Then $\tau(r) \leq \tau(e) + \tau(e^\prime) < 2\delta$, hence
$\|r f r\| < \vr$ for $f \in K$. However,
$f e = e^\prime f e = e^\prime (r f r) e$, hence
$\|f e\| \leq \|r f r\| < \vr$. An estimate for $\|e f\|$
is obtained similarly.
\end{proof}

\begin{proof}[Proof of Theorem \ref{t:main}]
It remains to establish (1) $\Rightarrow$ (2) for $1 < p < 2$.
Suppose $T \in B(L_p(\tau))$ is not strictly singular -- that is,
it fixes 
a subspace $X$. To show that $T$
is an isomorphism on $E \subset X$, so that $E$ is isomorphic to either $\ell_p$
or $\ell_2$, and both $E$ and $T(E)$ are complemented, we consider two cases
separately: (1) $\ball(X)$ is not $p$-equiintegrable (then $E \sim \ell_p$);
(2) $\ball(X)$ is $p$-equiintegrable (then $E \sim \ell_2$).

\begin{case}
 Let $\ball(X)$ be not $p$-equiintegrable.
By \cite[Theorem 5.1]{RX}, $X$ contains a complemented subspace $Y$,
isomorphic to $\ell_p$. Denote by $(f_n)$ an $\ell_p$-basis in $Y$.
$(Tf_n)$ is an $\ell_p$-basic sequence, hence, by \cite[Theorem 5.1]{RX}
again (or by \cite{Ran03}), there exists a normalized block sequence
$h_k = \sum_{j \in I_k} \alpha_j T f_j$, whose linear span is complemented
in $L_p(\tau)$. By \cite[Proposition 2.a.1]{LT1}, the linear span of
the vectors $g_k = T^{-1} h_k = \sum_{j \in I_k} \alpha_j f_j$ is complemented in
$\span[f_n : n \in \N]$, hence also in $L_p(\tau)$.
\end{case}

\begin{case} Suppose $\ball(X)$ is $p$-equiintegrable. By Corollary \ref{c:uncond},
$X$ contains a normalized unconditional basic sequence $(f_n)$. We shall use
$f_n$'s to produce the following sequence $(g_n)$:

\begin{enumerate}
\item
$\|g_n\|_p \in [1/2,2]$.
\item
$\sup_n \|g_n\|_\infty < \infty$.
\item
The sequence $(g_n)$ is weakly null in both $L_p(\tau)$ and $L_2(\tau)$.
\item
$(g_n)$ is equivalent to an orthonormal basis in $\ell_2$, in
both $L_p(\tau)$ and $L_2(\tau)$.
\item
$(T g_n)$ is equivalent to the $\ell_2$-basis.
\end{enumerate}

Without loss of generality, we assume $\|f_n\|_p = 1$ for every $n$.
Set $c = \inf_n \|T f_n\|$, and fix $\vr \in (0,\min\{1/40,c/(8\|T\|)\})$.
The sequence $(f_n)$ is $p$-equiintegrable, hence
there exists $\delta > 0$ such that $\max\{\|ef_n\|_p, \|f_ne\|_p\} < \vr$
whenever $e$ is a projection of trace not exceeding $\delta$.
Let $M = \delta^{-1} + 1$. Write $f_n = u_n |f_n|$, where $u_n$ is a
partial isometry from $(\ker f_n)^\perp$ onto $\overline{\ran f_n}$.
Let $\phi(t) = \left\{
 \begin{array}{ll}  t & t \leq M  \\  0 & t > M \end{array} \right.$,
and $\tilde{f}_n = u_n \phi(|f_n|)$.  Note that $\|f_n - \tilde{f}_n\|_p < \vr$.
Indeed, consider the spectral projection $r_n = \chi_{(M,\infty)}(|f_n|)$. Then
$1 = \| |f_n| \|_p^p \geq M^p \tau(r_n)$, hence $\tau(r_n) < \delta$.
By our choice of $\delta$, $\|f_n r_n\|_p = \| |f_n| r_n\|_p < \vr$.
Therefore, the $L_p(\tau)$ norm of
$f_n - \tilde{f}_n = u_n (|f_n| - \phi(|f_n|)) = u_n |f|_n r_n$
is less than $\vr$.

Passing to a subsequence, we can assume $\tilde{f}_n \to f$ weakly.
Then $\tilde{f}_n - f_n \to f$ weakly as well, hence
$\|f\| \leq \liminf \|\tilde{f}_n - f_n\| < \vr$.
Spectral calculus allows us to pick projections $q_1$ and $q_2$,
so that $\tau(q_1^\perp), \tau(q_2^\perp) < \delta$,
$f = q_1 f q_2 + q_1^\perp f q_2^\perp$, and
$N = \|q_1 f q_2\| < \infty$. Let $g_n = q_1(\tilde{f}_n - f)q_2$.
Then 
$$
f_n - g_n = q_1^\perp f_n + q_1 f_n q_2^\perp +
q_1(f - \tilde{f}_n)q_2 + q_1 f q_2 .
$$
The fact that $\tau(q_1^\perp), \tau(q_2^\perp) < \delta$ leads to
$$
\|f_n - g_n\|_p \leq \|q_1^\perp f_n\|_p + \|q_1 f_n q_2^\perp\|_p +
\|q_1(f - \tilde{f}_n)q_2\|_p + \|q_1 f q_2\|_p < 4 \vr ,
$$
and therefore, $\|g_n\|_p \subset [1-4\vr,1+4\vr] \subset [1/2,2]$.
We also have $\|g_n\|_\infty \leq \|\tilde{f}_n\| + \|q_1 f q_2\| \leq M+N$.
By H\"older's Inequality,
$$
\frac{1}{2} \leq \|g_n\|_p \leq
\|g_n\|_2 \leq \|g_n\|_p^{p/2} \|g_n\|_\infty^{1-p/2} \leq 2^\frac{p}{2}( M+N)^{1-\frac{p}{2}} . 
$$

Note that $g_n \to 0$ weakly in $L_p(\tau)$. That is, for any $x^* \in L_q(\tau)$
($1/p + 1/q = 1$), $\lim_n x^*(g_n) = 0$. As
$L_2(\tau) = \overline{L_q(\tau) \cap L_2(\tau)}^{\| \cdot \|_2}$,
$g_n \to 0$ weakly in $L_2(\tau)$ as well. Therefore, by passing to a subsequence
several times, and applying Proposition \ref{p:FDD}, we can assume that
the sequence $(g_n)$ is unconditional, both in $L_p(\tau)$ and in $L_2(\tau)$.
In $L_2(\tau)$, $(g_n)$ is equivalent to an orthonormal basis, hence,
for any finite sequence of scalars $(\alpha_n)$,
\begin{equation}
\|\sum_n \alpha_n g_n\|_p \leq \|\sum_n \alpha_n g_n\|_2 \in
\Big[ C_1 \big( \sum_i |\alpha_i|^2 \big)^{1/2} ,
C_2 \big( \sum_i |\alpha_i|^2 \big)^{1/2} \Big]
\label{eq:equiv l2}
\end{equation}
($C_1$ and $C_2$ are positive constants). On the other hand, 
by unconditionality, there exists
a constant $C_3$ so that, for any sequence $(\alpha_n)$,
$\|\sum_n \alpha_n g_n\|_p \geq C_3 \ave_\pm \|\sum_n \pm \alpha_n g_n\|_p$
(we are averaging over all possible signs). However, $L_p(\tau)$ has cotype $2$,
hence
$$
\ave_\pm \|\sum_n \pm \alpha_n g_n\|_p \geq
 2 C_4 \big(\sum_n |\alpha_n|^2 \|g_n\|_p^2\big)^{1/2} \geq
 C_4 \big(\sum_n |\alpha_n|^2\big)^{1/2} ,
$$
for some $C_4$, and therefore, $\|\sum_n \alpha_n g_n\|_p \geq C_3 C_4
\big(\sum_n |\alpha_n|^2\big)^{1/2}$. Together with \eqref{eq:equiv l2},
this shows that $(g_n) \subset L_p(\tau)$ is equivalent to the $\ell_2$-basis.

Next show that the sequence $(T g_n)$ is equivalent to the $\ell_2$-basis.
We have
$$
\|T g_n\|_p \geq \|T f_n\|_p - \|T\| \|f_n - g_n\|_p \geq
c - 4 \vr \|T\| > \frac{c}{2} .
$$
The sequence $(T g_n) \subset L_p(\tau)$ is weakly null, hence,
by passing to a further subsequence, we can assume it is unconditional.
Combining unconditionality with the cotype $2$ property of $L_p(\tau)$, we
conclude the existence of a constant $C_5$ so that, for any finite sequence
of scalars $(\alpha_n)$, $\|\sum_n \alpha_n T g_n\| \geq C_5 (\sum_n |\alpha_n|^2)^{1/2}$.
On the other hand, $\|\sum_n \alpha_n T g_n\| \leq \|T\| \|\sum_n \alpha_n g_n\|
 \leq C_6 (\sum_n |\alpha_n|^2)^{1/2}$, for some constant $C_6$.

By Proposition \ref{p:complem},
we can assume that $T$ is an isomorphism on a complemented subspace
$Y$, isomorphic to $\ell_2$. Using Proposition \ref{p:complem} again,
we can assume that $T(Y)$ is complemented as well.
\end{case}
\end{proof}

The proof of Theorem~\ref{t:main} can be modified to yield the following:

\begin{theorem}\label{t:main'}
Suppose $\tau_1, \tau_2$ are finite traces on hyperfinite von Neumann algebras,
and $p_1, p_2 \in (1, \infty)$ are distinct, with either $p_1 \geq 2$,
or $p_2 > p_1$. Then, for $T$ in $B(L_{p_1}(\tau_1), L_{p_2}(\tau_2))$, the
following statements are equivalent: $(i)$  $T$ is not strictly singular;
$(ii)$ $T$ is an isomorphism on $E$, where $E$  is isomorphic to $\ell_2$,
and both $E$ and $T(E)$ are complemented.
\end{theorem}

Note that the statement of this theorem cannot be extended to the case of
$1 < p_2 < p_1 < 2$, even in the commutative setting: it is well known
that, in this case, $L_{p_2}(0,1)$ contains an uncomplemented copy of $\ell_{p_1}$.

\begin{proof}[Sketch of a proof]
We follow the proof of Theorem \ref{t:main}, preserving the notation.
Suppose $T \in L_{p_1}(\tau_1) \to L_{p_2}(\tau_2)$ is an isomorphism
on a subspace $X \subset L_{p_1}(\tau_1)$, and show that $L_{p_1}(\tau_1)$
contains a subspace $E$, isomorphic to $\ell_2$, so that $E$ and $T(E)$ are
complemented in $L_{p_1}(\tau_1)$ and $L_{p_2}(\tau_2)$, respectively.

First consider the case of $p_1 \geq 2$.
By passing to a subspace, we can assume that $X$ is isomorphic to
either $\ell_{p_1}$ or $\ell_2$, and complemented. But $L_{p_2}(\tau_2)$
cannot contain a copy of $\ell_{p_1}$. Indeed, for $p_2 \geq 2$,
\cite[Proposition 5.4]{RX} shows that every unconditional sequence
in $L_{p_2}(\tau)$ contains a subsequence equivalent to either the
$\ell_2$-basis, or the $\ell_{p_2}$-basis, which are both different from
the $\ell_{p_1}$-basis. For $p_2 \leq 2$, we know (see e.g. \cite{PX})
that $L_{p_2}(\tau_2)$ has cotype $2$, while $\ell_{p_1}$ does not.
So, we can assume $X$ is complemented, and isomorphic to $\ell_2$.
If $p_2 \geq 2$, then $L_{p_2}(\tau_2)$ has type $2$ \cite{PX}, hence,
by Maurey Extension Theorem (see e.g. \cite[Theorem 12.22]{DJT}),
$T(X)$ is complemented in $L_{p_2}(\tau_2)$ (thus, we can take $E = X$).
For $p_2 < 2$, we use Proposition \ref{p:complem} to obtain $E \subset X$
with the desired properties.

Now suppose $1 < p_1 < 2 \leq p_2 < \infty$, and $p_1 \neq p_2$.
Then $T(X)$ contains either a copy of $\ell_2$, or a copy of $\ell_{p_2}$.
However, $\ell_{p_2}$ cannot embed into $L_{p_1}(\tau_1)$, hence,
passing to a subspace, we can assume that $X$ is isomorphic to $\ell_2$.
By Maurey Extension Theorem, any copy of $\ell_2$ in $L_{p_2}(\tau_2)$
is complemented. Proposition \ref{p:complem} allows us to find a subspace
$E \subset X$, complemented in $L_{p_1}(\tau_1)$.

In the case of $1 < p_1 < p_2 \leq 2$,
note first that $\ball(X)$ is $p_1$-equiintegrable. Indeed, otherwise
$X$ contains a copy of $\ell_{p_1}$. However, $L_{p_2}(\tau_2)$
cannot contain a copy of $\ell_{p_1}$, since the former space has type
$p_2$, and the latter does not.

Now find a normalized unconditional basic sequence $(f_n)$ in $\ball(X)$,
and use it (as in Theorem \ref{t:main}) to obtain a sequence $(g_n)$ in
$L_{p_1}(\tau_1)$ so that:
(1) $\|g_n\|_{p_1} \in [1/2,2]$.
(2) $\sup_n \|g_n\|_\infty < \infty$.
(3) The sequence $(g_n)$ is weakly null in both $L_{p_1}(\tau)$ and $L_2(\tau)$.
(4) $(g_n)$ is equivalent to an orthonormal basis in $\ell_2$, in
both $L_{p_1}(\tau)$ and $L_2(\tau)$.
(5) $(T g_n)$ is equivalent to the $\ell_2$-basis.
\end{proof}

In a similar manner we prove:

\begin{proposition}\label{p:SS=K}
Suppose $\infty > p_1 \geq 2 \geq p_2 > 1$, and $\tau_1, \tau_2$
are finite traces on hyperfinite von Neumann algebras. Then
$\iSS(L_{p_1}(\tau_1), L_{p_2}(\tau_2)) =
 \iK(L_{p_1}(\tau_1), L_{p_2}(\tau_2))$.
\end{proposition}

Clearly the conclusion fails if $p_1$ and $p_2$ are ``on the same side of $2$''.
Indeed, if $p_1, p_2 \geq 2$, and $p_1 \neq p_2$, then $L_{p_1}$ ($L_{p_2}$)
contains a complemented copy of $\ell_2$ (resp. $\ell_{p_2}$), and the formal
identity from $\ell_2$ to $\ell_{p_2}$ is strictly singular, but not compact.
The case of $p_1, p_2, \leq 2$ and $p_1 \neq p_2$ can be handled similarly. 
However, if we restrict ourselves to positive operators, more can be said,
see Theorem \ref{t:positive}.

\begin{proof}
Consider a non-compact $T \in B(L_{p_1}(\tau_1), L_{p_2}(\tau_2))$.
Then there exists a weakly null sequence $(x_k)$ in $L_{p_1}(\tau_2)$
so that $(T x_k)$ is not null. By passing to a subsequence if necessary,
we can assume that the sequences $(x_k)$ and $(T x_k)$ are unconditional,
and $\inf_k \|T x_k\| > 0$.
By \cite[Proposition 5.4]{RX}, by passing to a further subsequence if
necessary, we can assume that $(x_k)$ is equivalent to either the $\ell_2$-basis,
or the $\ell_{p_1}$ basis. In particular, there exists a constant $c_1$
so that the inequality
$\|\sum_k \alpha_k T x_k\|_{p_1} \leq c_1 \big( \sum_k |\alpha_k|^2 \big)^{1/2}$
holds for any finite sequence of scalars $(\alpha_k)$.
On the other hand, $L_{p_2}(\tau_2)$ has cotype $2$ (see e.g. \cite{PX}),
hence there exists a constant $c_2$ so that,
for any finite sequence $(\alpha_k)$,
$c_2 \|\sum_k \alpha_k T x_k\|_{p_2} \geq \big( \sum_k |\alpha_k|^2 \big)^{1/2}$.
Consequently,
$\|\sum_k \alpha_k T x_k\|_{p_1} \leq c_1 c_2 \|\sum_k \alpha_k T x_k\|_{p_2}$,
which contradicts the strict singularity of $T$.
\end{proof}



\subsection{Strict singularity and compactness of positive operators}


For positive operators, some results of the previous subsection can be
sharpened. The proposition below refines Proposition \ref{p:SS=K}.

\begin{theorem}\label{t:positive}
Suppose $\tau_1$ and $\tau_2$ are normal faithful finite traces on
hyperfinite von Neumann algebras $\A_1$ and $\A_2$, respectively,
and $1 < u_2 \leq u_1 < \infty$. Then
$\iSS(L_{u_1}(\tau_1),L_{u_2}(\tau_2))_+ =
 \iK(L_{u_1}(\tau_1),L_{u_2}(\tau_2))_+$.
\end{theorem}



\begin{remark}\label{r:pos optimal}
In the commutative case, related results were obtained in \cite{CG87, Fl04}.
The examples from \cite{Fl04} show that the conclusion of Theorem \ref{t:positive}
fails, even in the commutative case, when $u_1 < u_2$. We cannot extend our
theorem to the case of $1 \leq u_2 < u_1 = \infty$ either.
Indeed, consider the formal identity $i : L_\infty(0,1) \to L_{u_2}(0,1)$.
It is well known that the norms $\| \cdot \|_\infty$ and $\| \cdot \|_q$
cannot be equivalent on an infinite dimensional subspace for $q < \infty$,
hence $i$ is strictly singular. However, $i$ is not compact, since all
Rademacher functions belong to $i(\ball(L_\infty))$.

Theorem \ref{t:positive} also fails for $u_1 = u_2 = 1$, even in the commutative case:
there exists a positive non-compact strictly singular operator on $L_1$. Indeed,
let $(r_n)_{1}^{\infty}$ be a Rademacher system and $e$ be the identity.
Define $x_n=e+r_n$. 
Set $U : \ell_1 \to L_1$ as $U \delta_n=x_n$, where $(\delta_n)$
is the canonical basis for $\ell_1$. It is easy to check that $U$
is positive, and not compact. By Khintchine Inequality,
$\span[e, r_1, r_2, \ldots]$ is isomorphic to $\ell_2$, hence
the same is true for $\span[x_n : n \in \N]$. Therefore, $U$ is
strictly singular. The required operator is
the composition of a positive projection on a copy of $\ell_1$ with $U$.
\end{remark}

The rest of this subsection is devoted to proving Theorem \ref{t:positive}.
We assume that all traces are normalized.
Denote by $L_p(\tau)_{sa}$ the self-adjoint (real)
part of $L_p(\tau)$. For the proof we need an auxiliary result.

\begin{lemma}\label{l:dichotomy}
Suppose $2 < p < \infty$, and $(x_k)$ is an unconditional self-adjoint
normalized sequence in $L_p(\tau)$, where $\tau$ is a finite normal
faithful trace on a von Neumann algebra $\A$. Then either $(x_k)$ is equivalent
to the $\ell_2$-basis, or there exist $n_1 < n_2 < \ldots$, and a sequence
of mutually orthogonal projections $p_k \in \A$, so that
$\lim_k \|x_{n_k} - p_k x_{n_k} p_k\|_p = 0$.
\end{lemma}

\begin{proof}
The proof uses a variation on a well-known ``Kadec-Pelczynski'' method.
Our exposition follows \cite{S96}. For $c > 0$, set
$$
M_c = \big\{x \in L_p(\tau) : \tau(\chi_{(c\|x\|_p, \infty)}(|x|)) \geq c \big\} .
$$
If there exists $c > 0$ so that $x_k \in M_c$ for every $k$, then, by
the proof of \cite[Theorem 2.4]{S96}, $(x_k)$ is equivalent to the
$\ell_2$-basis. Otherwise, by passing to a subsequence, we can assume
that the projections
$$
q_k = \chi_{\R \backslash (-4^{-k}, 4^{-k})}(x_k)
$$
satisfy two conditions:
\begin{enumerate}
\item
$\tau(q_1) < 1/8$, and $\tau(q_k) < \tau(q_{k-1})/8$ for $k > 1$.
\item
If $q$ is a projection with $\tau(q) \leq 2 \tau(q_k)$, then
$\max_{i < k} \|q x_i\|_p < 4^{-(k+1)}$ (see \cite[Theorem 4.2]{FK}
to show that this can be satisfied).
\end{enumerate}
Let $r_k = \vee_{j > k} q_j$, and $p_k = q_k \wedge r_k^\perp$.
We claim that
\begin{equation}
\|x_k - p_k x_k p_k\|_p < 4^{-k} .
\label{eq:ineq}
\end{equation}
Indeed, write $p_k = q_k - q_k^\prime$. Then
$$
x_k - p_k x_k p_k = (x_k - q_k x_k q_k) + q_k^\prime x_k (q_k - q_k^\prime)
+ q_k x_k q_k^\prime .
$$
Clearly, $\|x_k - q_k x_k q_k\|_p \leq \|x_k q_k^\perp\|_\infty < 8^{-k}$.
Furthermore, $\tau(r_k) \leq \sum_{j > k} \tau(q_j) < 2 \tau(q_{k+1})$, hence,
by Kaplansky's Formula \cite[Theorem 6.1.7]{KR2},
$$
\tau(q_k^\prime) = \tau(q_k) - \tau(q_k \wedge r_k^\perp) =
\tau(q_k \vee r_k^\perp) - \tau(r_k^\perp) \leq
1 - \tau(r_k^\perp) = \tau(r_k) < 2 \tau(q_{k+1}) .
$$
Thus, $\|q_k^\prime x_k\|_p = \|x_k q_k^\prime\|_p < 4^{-(k+1)}$.
Together, these inequalities give us \eqref{eq:ineq}.
\end{proof}

\begin{lemma}\label{pos_eqi}
Suppose $\tau$ is a faithful normal semifinite trace on a von Neumann algebra,
and $1 \leq p < \infty$.
Then every $p$-equiintegrable weakly null sequence  $(f_n) \subset L_p(\tau)_+$
is norm null. In particular, no sequence in $L_p(\tau)_+$
is equivalent to a standard basis of $\ell_2$.
\end{lemma}

\begin{proof}
Consider a weakly null sequence $(f_j) \subset L_p(\tau)_+$. Then
$\lim_n \tau(f_j) = 0$. The case of $p=1$
is the easiest to handle: $\|f_j\|_1 = \tau(f_j) \to 0$.

Now let $1 < p < \infty$.
Suppose, for the sake of contradiction, that $(f_j)$ is not a norm null sequence.
Without loss of generality assume that $(f_j)$ is normalized.
Fix $0<c<1$. Since $(f_n)$ is $p$-equiintegrable, by Lemma~\ref{l:equi} there exists $C > 0$ so that, for any $j$, $\|f_j - w_j\|_p < c $,
where $w_j = {\boldsymbol{\varphi}}_C(f_j)$ (the function
${\boldsymbol{\varphi}}_C(t)$ is defined as $\min\{t,C\}$).
By H\"older Inequality,
$\|w_j\|_{p} \leq \|w_j\|_1^{1/p} \|w_j\|_\infty^{1-1/p}$, hence
$$
\tau(f_j) \ge \tau(w_j) = \|w_j\|_1 \geq
 \|w_j\|_{p}^{p} \|w_j\|_\infty^{1-p} > (1-c)^{p} C^{1-p} .
$$
This contradicts $\lim_j \tau(f_j) = 0$.
\end{proof}

\begin{proof}[Proof of Theorem \ref{t:positive}]
We have to show that any strictly singular positive
$T \in B(L_{u_1}(\tau_1),L_{u_2}(\tau_2))$ is compact.
The case of $u_2 \leq 2 \leq u_1$ follows from Proposition \ref{p:SS=K},
so only the case of $u_1$ and $u_2$ being ``on the same side of $2$'' remains.
Throughout the proof, for $s \in \{1,2\}$, we write $v_s = u_s/(u_s-1)$
(that is, $1/u_s + 1/v_s = 1$). Then $v_2 \geq v_1$.

First consider $2 \leq u_2 \leq u_1 < \infty$.
Without loss of generality, we can assume $\|T\| \leq 1$.
Suppose, for the sake of contradiction, that $T$ (or equivalently, $T^*$) is not
compact. Note that $T$ maps $L_{u_1}(\tau_1)_{sa}$ into $L_{u_2}(\tau_2)_{sa}$.
Then there exists a weakly null normalized sequence $(x_k)$ in
$L_{u_1}(\tau_1)_{sa}$, so that $\|T x_k\| > 6 c > 0$ for any $k$.
By passing to a subsequence twice, and invoking Corollary \ref{c:uncond}, 
we can assume that the sequences $(x_k)$ and $(T x_k)$ are unconditional.
Furthermore, by \cite[Proposition 5.4]{RX}, $(x_k)$ ($(Tx_k)$) is equivalent
either to the $\ell_2$-basis, or to the $\ell_{u_1}$-basis (respectively,
either to the $\ell_2$-basis, or to the $\ell_{u_2}$-basis).
As $T$ is bounded,
and strictly singular, only one possibility is open to us: $(x_k)$ and $(T x_k)$ are
equivalent to the $\ell_2$-basis, and the $\ell_{u_2}$-basis, respectively.

Applying Lemma \ref{l:dichotomy} (and passing to a subsequence again
if necessary), we conclude that there exist mutually
orthogonal projections $p_k$ so that $\|T x_k - y_k\| < 100^{-k} c$,
where $y_k = p_k (T x_k) p_k$. Find a sequence of positive norm one
elements $z_k \in L_{u_2}^*(\tau_2) = L_{v_2}(\tau_2)$,
so that $z_k = p_k z_k p_k$ for any $k$, and
$|\tau_2(z_k y_k)| \geq \|y_k\|/2 > 5c/2$.
Then, by passing to a subsequence, we may assume that the inequality
\begin{eqnarray*}
\|T^*z_j\| & \ge & \big| \tau_1\big((T^* z_j) \cdot x_j\big)\big| =
    \big| \tau_2(z_j \cdot Tx_j)\big|  \\
                 & \ge & \Big| | \tau_2(z_j y_j)|- | \tau_2(z_j \cdot (Tx_j-y_j)| \Big|
 >2c
\end{eqnarray*}
holds for any $j$.
Note that the sequence $(z_j)$ is equivalent to the $\ell_{v_2}$-basis,
hence weakly null. The sequence $(T^* z_j)$ is weakly null as well.
By passing to a subsequence if necessary, we can assume $(T^* z_j)$ is
unconditional. This sequence has no subsequences equivalent
to the $\ell_{v_1}$-basis. Indeed, if there existed such a subsequence,
then $T$ would not be strictly singular ($v_1 = v_2$), or bounded
($v_2 > v_1$). By \cite[Theorem 5.3]{RX}, $(T^* z_j)$ is
$v_1$-equiintegrable. This contradicts Lemma~\ref{pos_eqi}.

Now suppose $1 < u_2 \leq u_1 \leq 2$.
Consider $T \in \iSS(L_{u_1}(\tau_1),L_{u_2}(\tau_2))_+$.
Note that $L_{u_1}(\tau_1)^* = L_{v_1}(\tau_1)$ is subprojective
(that is, any infinite dimensional subspace $X \subset L_{v_1}(\tau_1)$ contains
a further infinite dimensional subspace $Y$, complemented in $L_{v_1}(\tau_1)$),
hence, by \cite[Theorem 7.54(ii)]{Ai}, $T^*$ is strictly singular.
By the above, $T^*$ is compact, hence so is $T$.
\end{proof}

\section{Operators on Schatten-Lorentz spaces}\label{s:lore}

In this section, we consider operators on Schatten spaces $\se$, where
$\ce = \lore(w,p)$ is a Lorentz space (defined in Section \ref{s:intro}).
For more information,
the reader is referred to \cite[Section 4.e]{LT1}. In particular, we
will often use \cite[Proposition 4.e.3]{LT1}: $\lore(w,p)$ is saturated
with complemented block subspaces, isomorphic to $\ell_p$.

\begin{theorem}\label{thm:lor SS=IN}
For $T \in B(\cs_{\lore(w,p)})$, the following statements are equivalent:
\begin{enumerate}
 \item
  There exists a subspace $E$ of $\cs_{\lore(w,p)}$, isomorphic
  to either $\ell_2$ or $\ell_p$, so that $T|_E$ is an isomorphism, and
  both $E$ and $T(E)$ are complemented.
 \item
  $T$ is not strictly singular.
 \item
  $T$ is not inessential.
\end{enumerate}
\end{theorem}

The following lemmas may be known, but we haven't been able to find them in the
literature.

\begin{lemma}\label{l:twice compl}
Suppose $X$ and $Y$ are Banach spaces, and $X_0$ is a complemented subspace
of $X$ isomorphic to $\ell_p$ ($1 \leq p < \infty$) or $c_0$. Suppose,
furthermore, that $T$ is an isomorphism from $X_0$ onto $Y_0 \subset Y$,
and $Y_0$ contains a subspace $Y_1$, complemented in $Y$. Then $X_0$ contains
a subspace $E$, isomorphic to $\ell_p$ (or $c_0$), so that both $E$ and $T(E)$
are complemented.
\end{lemma}

\begin{proof}[Sketch of a proof]
We work with $\ell_p$ spaces, since the case of $c_0$ is handled similarly.
Let $(x_j)$ be a basis in $X_0$, equivalent to the $\ell_p$-basis.
Then,  the vectors $y_j = T(x_j)$  form a basis in $Y_0$, once again equivalent to
the $\ell_p$-basis. Let $\kappa$ be the basis constant of $(y_j)$, and
$\vr_k = (2 + 2\kappa)^{-1} 10^{-k}$ ($k \in \N$).
By the ``gliding hump'' argument (cf. the proof of \cite[Proposition 1.a.11]{LT1}),
there exists a normalized sequence $(v_k)$, of the
form $v_k = \sum_{j=N_k+1}^{N_{k+1}} \alpha_j y_j$, so that, for any $k$,
there exists $v_k^\prime \in Y_1$, satisfying $\|v_k - v_k^\prime\| < \vr_k$.
By \cite[Section 2.a]{LT1}, $F = \span[v_k : k \in \N]$ is complemented in $Y_0$,
and isomorphic to $\ell_p$.
By the Small Perturbation Principle, $F^\prime = \span[v_k^\prime : k \in \N]$
is complemented in $Y_0$. Combining the projections from $Y$ onto $Y_1$, and from
$Y_1 \subset Y_0$ onto $F^\prime$, we see that $F^\prime$ is complemented in $Y$.
Another application of the  Small Perturbation Principle shows that $F$ is complemented
in $Y$ as well. The space $E = T^{-1}(F)$ is spanned by the vectors
$\sum_{j=N_k+1}^{N_{k+1}} \alpha_j x_j$, forming a block basis in $X_0$.
Thus, $E$ is complemented in $X_0$, hence also in $X$.
\end{proof}

\begin{lemma}\label{l:l_2+lor}
Any subspace $E$ of $\ell_2 \oplus \lore(w,p)$ contains a further subspace $F$,
which is isomorphic either to $\ell_2$ or $\ell_p$, and complemented in
$\ell_2 \oplus \lore(w,p)$.
\end{lemma}

\begin{proof}
For the sake of brevity, let $X = \ell_2$ and $Y = \lore(w,p)$. Denote by $P_X$ and $P_Y$
the canonical projections from $X \oplus Y$ onto $X$ and $Y$, respectively.

Suppose first that $E$ contains a subspace $G$ so that $P_X|_G$ is an isomorphism.
Then $G$ is a Hilbert space. Moreover, $G$ is complemented in $X \oplus Y$. Indeed,
let $Q$ be the orthogonal projection from $X$ onto $P_X(G)$. Then
$P = P_X^{-1} P P_X$ is a projection onto $G$.

Now suppose $P_X|_E$ is strictly singular. Then $E$ contains a subspace $G$, so that
$P_X|_G$ is compact, and has norm less than $1/2$. Then $P_Y|_G$ is an isomorphism.
$G^\prime = P_Y(G)$ is a closed subspace of $Y$, hence it must contain a subspace $F^\prime$,
isomorphic to $\ell_p$, and complemented in $Y$ by a projection $Q$. Then
$F = P_Y^{-1}(F^\prime) \subset G \subset E$ is isomorphic to $\ell_p$, and
complemented in $X \oplus Y$ (one can use the projection $P_Y^{-1} Q P_Y$).
\end{proof}

\begin{proof}[Proof of Theorem \ref{thm:lor SS=IN}]
The implications $(1) \Rightarrow (3) \Rightarrow (2)$ are clear.
It remains to establish $(2) \Rightarrow (1)$. For the sake of brevity,
set $\ce = \lore(w,p)$. 

Suppose $T$ is not strictly singular -- that is, $T$ fixes $X \subset \se$.
\cite{OS_SP} shows that, by passing to a subspace if necessary,
we can assume that $X$ embeds into a complemented copy of $\ell_2 \oplus \ce$.
By Lemma \ref{l:l_2+lor}, we can furthermore assume that $X$ is complemented, and
isomorphic to either $\ell_2$ or $\ell_p$. An application of Lemma \ref{l:twice compl}
finishes the proof.
\end{proof}

\begin{remark}\label{r:lore p_1 p_2}
One can similarly prove that for any $T \in  B(\cs_{w_1,p_1}, \cs_{w_2,p_2})$
$(p_1 \neq p_2)$, $T$ is not strictly singular iff $T$ fixes $E$,
isomorphic to $\ell_2$, so that both $E$ and $T(E)$ are complemented.
\end{remark}

Next we show that, on Schatten-Lorentz spaces, the ideals of compact,
strictly singular, and finitely strictly singular operators need not coincide.

\begin{proposition}\label{p:FSS not SS}
Suppose $1 \leq p_1, p_2 < \infty$. If either $p_1 < \max\{2,p_2\}$, or $p_2 > 2$,
then
$$
\iK(\cs_{\lore(w_1,p_1)}, \cs_{\lore(w_2,p_2)})
 \subsetneq \iFSS(\cs_{\lore(w_1,p_1)}, \cs_{\lore(w_2,p_2)})
 \subsetneq \iSS(\cs_{\lore(w_1,p_1)}, \cs_{\lore(w_2,p_2)}) .
$$
\end{proposition}

Note that the above conditions on $p_1$ and $p_2$ fail if and only if
$p_1 \geq 2 \geq p_2$. This result is sharp, as shown by Theorem \ref{t:C_p->C_q}
in the case when $w_1(k) = w_2(k) = 1$ for every $k$.

For $p_1 = p_2 = p$, Proposition \ref{p:FSS not SS} yields:

\begin{corollary}\label{c:FSS not SS}
If $p \in[1,\infty) \backslash \{2\}$, then
$\iFSS(\cs_{\lore(w_1,p)}, \cs_{\lore(w_2,p)})
 \subsetneq \iSS(\cs_{\lore(w_1,p)}, \cs_{\lore(w_2,p)})$.
\end{corollary}

The case of $p=2$ is excluded for a reason. Suppose, for instance, that
$w_1(k) = w_2(k) = 1$, for every $k$. Then the spaces
$\cs_{\lore(w_1,2)}$ and $\cs_{\lore(w_2,2)}$ are isometric to $\ell_2$,
and we know that $\iSS(\ell_2) = \iFSS(\ell_2) = \iK(\ell_2)$.

The following lemma (needed to prove Proposition \ref{p:FSS not SS}) may be
of independent interest.

\begin{lemma}\label{l:l_p(l_2)}
For $1 \leq p < \infty$, $\cs_{\lore(w,p)}$ contains a complemented
isomorphic copy of $\ell_p(\ell_2)$.
\end{lemma}

\begin{proof}
We view $\ell_p(\ell_2)$ as the space of all sequences $(c_{ij})$, with
$$
\|(c_{ij})\| = \Big(\sum_i \big( \sum_j |c_{ij}|^2 \big)^{p/2} \Big)^{1/p} .
$$
For brevity, set $\ce = \lore(w,p)$. To show that $\se(\N \times \N \times \N)$
contains a complemented copy of $\ell_p(\ell_2)$,
consider a normalized block sequence $(u_k) \subset \ce$, equivalent to the
$\ell_p$-basis, and complemented in $\ce$ via a projection $Q$ (the existence
of such $(u_k)$ is guaranteed by \cite[Proposition 4.e.3]{LT1}). Write
$u_k = \sum_{s=N_k+1}^{N_{k+1}} \alpha_s \delta_s$, where $(\delta_s)$ is the
canonical basis for $\lore(w,p)$. Without loss of generality, assume $\alpha_s \geq 0$.
Consider
$$
X = \span[E_{kk} \otimes E_{1j} \otimes v_k :
  k, j \in \N] \subset \se(\N \times \N \times \N)] ,
{\textrm{  where  }}
v_k = \sum_{s=N_k+1}^{N_{k+1}} \alpha_s E_{ss} \in \se .
$$
Then the map
$$
T : \ell_p(\ell_s) \to X : (c_{kj})_{k,j \in \N} \mapsto
\sum_{k,j} c_{kj} E_{kk} \otimes E_{1j} \otimes v_k
$$
is an isomorphism. Indeed,
$$
\big(T((c_{kj})_{k,j \in \N})\big) \big(T((c_{kj})_{k,j \in \N})\big)^\star =
\sum_k E_{kk} \otimes \big(\sum_j |c_{kj}|^2 E_{11} \big) \otimes v_k v_k^\star ,
$$
hence
$$   \eqalign{
\|T((c_{kj})_{k,j \in \N})\|_{\ce}
&
=
\|\sum_k E_{kk} \otimes \big(\sum_j |c_{kj}|^2 \big)^{1/2} E_{11} \otimes v_k\|
\cr
&
=
\|\sum_k \big(\sum_j |c_{kj}|^2 \big)^{1/2} u_k\| \sim
\|(c_{kj})\|_{\ell_p(\ell_2)} .
}  $$

To define the projection onto $X$, note first that, for any symmetric
sequence space $\ce$, the space
$$
Y = \span[E_{kk} \otimes E_{1j} \otimes E_{ss} : k, j, s \in \N]
\subset \se(\N \times \N \times \N)
$$
is contractively complemented in $\se(\N \times \N \times \N)$, via
the projection $Q_d \otimes Q_r \otimes Q_d$, where $Q_d$ and $Q_r$ denote the
coordinate projections from $\se$ onto $\span[E_{kk} : k \in \N]$ and
$\span[E_{1j} : j \in \N]$, respectively. Indeed, $Q_d \otimes I_{\se}$
is contractive, by a standard ``diagonalization''  argument (also known
as ``pinching'', or ``Tong's trick'', see e.g. the proof of
\cite[Proposition 1.c.8]{LT1}). Furthermore, $Q_r \otimes I_{\se}$ is contractive:
for $x \in \se(\N \times \N)$, $Q_r \otimes I_{\se} x = p x$, where $p$ is the
orthogonal projection onto $\span[e_k \otimes e_1]$, ($(e_s)$ being the canonical
basis for $\ell_2$). Therefore,
$$
Q_d \otimes Q_r \times Q_d =
(Q_d \otimes I_{\se} \otimes I_{\se})
(I_{\se} \otimes Q_r \otimes I_{\se})
(I_{\se} \otimes I_{\se} \otimes Q_d)
$$
is contractive as well.

Next construct a projection $P$ from $Y$ onto $X$.
Let $R$ be a projection from $\lore(w,p)$ onto $\span[u_k : k \in \N]$.
By a ``diagonalization'' argument again, we can assume that
$R \delta_s \in \span[u_k]$ whenever $s \in [N_k+1, N_{k+1}]$.
Denoting by $(\delta_s^*) \in \ce^*$ the functionals biorthogonal to $(\delta_s)$,
we can find (as in \cite{CKT})
$w_k = \sum_{s = N_k+1}^{N_{k+1}} \beta_s \delta_s^* \in \ce^*$
so that, for $x \in \ce$, $Rx = \sum_k \langle w_k , x \rangle u_k$.
Define, for $(x_{kj}) \subset \se$,
$$
P\big(\sum_k E_{kk} \otimes \sum_j E_{1j} \otimes x_{kj}\big) =
\sum_k E_{kk} \otimes \sum_j E_{1j} \otimes \langle w_k, x_{kj} \rangle v_k .
$$
It remains to show that $P$ is bounded. To this end, consider 
$$
x = \sum_k E_{kk} \otimes \sum_j E_{1j} \otimes x_{kj} .
$$
Note that
$$
\big(Px\big) \big(Px\big)^\star =
\sum_k E_{kk} \otimes E_{11} \otimes \sum_j |\langle w_k, x_{kj} \rangle|^2 v_k^2 ,
$$
hence
$$  \eqalign{
\|Px\|_{\se}
&
=
\big\|\sum_k (\sum_j |\langle w_k, x_{kj} \rangle|^2)^{1/2} u_k \big\|_{\ce}
\cr
&
=
\big\|\sum_k (\sum_j \lambda_{kj} \langle w_k, x_{kj} \rangle) u_k \big\|_{\ce} =
\big\|\sum_k \langle w_k, \sum_j \lambda_{kj}x_{kj} \rangle u_k \big\|_{\ce} ,
}  $$
where
$$
\lambda_{kj} = \frac{\overline{\langle w_k, x_{kj} \rangle}}{
 (\sum_j |\langle w_k, x_{kj} \rangle|^2)^{1/2}}
$$
(with the convention $0/0 = 0$).
Next identify $\ce$ with $\ce(\N \times \N)$, and denote the standard
symmetric basis in the latter space by $(\delta_i \otimes \delta_j)_{i,j \in \N}$. Then
$u_k$ can be identified with $\tilde{u}_k = \sum_i \alpha_i \delta_k \otimes \delta_i$.
Set $\tilde{w}_k = \sum_i \beta_i \delta_k^* \otimes \delta_i^*$.
In this setting, there arises a projection $\tilde{R}$ onto
$\span[\tilde{u}_k : k \in \N]$, defined via
$$
\tilde{R} (\sum_{kj} \gamma_{kj} \delta_k \otimes \delta_j) =
\sum_k \langle \tilde{w}_k , \sum_j \gamma_{kj} \rangle \tilde{u}_k .
$$
Moreover, $\|R\| = \|\tilde{R}\|$.
For $(x_{kj})$ as above, set $\tilde{x}_{kj} = \delta_k \otimes x_{kj}$.
In other words, if $x_{kj} = \sum_s x_{kj}^{(s)} \delta_s$, then
$\tilde{x}_{kj} = \sum_s x_{kj}^{(s)} \delta_k \otimes \delta_s$. Then
$$  \eqalign{
&
\big\|\sum_k \langle w_k, \sum_j \lambda_{kj}x_{kj} \rangle u_k \big\|_{\ce} =
\big\|\sum_k \langle \tilde{w}_k,
 \sum_j \lambda_{kj} \tilde{x}_{kj} \rangle \tilde{u}_k \big\|_{\ce(\N \times \N)}
\cr
&
=
\big\| \tilde{R}
 \big(\sum_k (\sum_j \lambda_{kj} \tilde{x}_{kj}) \big) \|_{\ce(\N \times \N)} \leq
\|\tilde{R}\| \|\sum_k (\sum_j \lambda_{kj} \tilde{x}_{kj})\|_{\ce(\N \times \N)} .
}  $$
Now consider
$$
y = \sum_k E_{kk} \otimes (\sum_j \lambda_{kj} \otimes E_{j1}) \otimes I_{\ell_2}
  \in B(\ell_2(\N \times \N \times \N)) .
$$
Note that, for any $k$, $\sum_j |\lambda_{kj}|^2 \leq 1$, hence
$y$ is a contraction. Then
$$
\sum_k E_{kk} \otimes E_{11} \otimes (\sum_j \lambda_{kj} x_{kj}) = xy ,
$$
hence,
$$
\big\|\sum_k (\sum_j \lambda_{kj} \tilde{x}_{kj})\big\|_{\ce(\N \times \N)} =
\big\|\sum_k E_{kk} \otimes E_{11} \otimes (\sum_j \lambda_{kj} x_{kj})\big\|_{\se}
\leq \|x\|_{\se} ,
$$
which shows that $P$ is bounded.
\end{proof}

\begin{proof}[Proof of Proposition \ref{p:FSS not SS}]
First show that
$\iFSS(\cs_{\lore(w_1,p_1)}, \cs_{\lore(w_2,p_2)})$ is a proper subset of
$\iSS(\cs_{\lore(w_1,p_1)}, \cs_{\lore(w_2,p_2)})$. Recall first
(cf. \cite{Pli}) that, for $u < v$, the formal identity from
$(\oplus_n \ell_2^n)_u$ to $(\oplus_n \ell_2^n)_v$ is strictly singular,
but not finitely strictly singular. Now use Lemma \ref{l:l_p(l_2)}.
If $p_1 < 2$, then $\cs_{\lore(w_1,p_1)}$ contains a complemented
copy of $(\oplus_n \ell_2^n)_{p_1}$, and the formal identity from
$(\oplus_n \ell_2^n)_{p_1}$ to $(\oplus_n \ell_2^n)_2 = \ell_2$.
In the case of $p_1 < p_2$, the reasoning is similar, except that now, we use
the formal identity from $(\oplus_n \ell_2^n)_{p_1}$ to $(\oplus_n \ell_2^n)_{p_2}$.
Finally, for $p_2 > 2$, we map $(\oplus_n \ell_2^n)_2 = \ell_2$ to
$(\oplus_n \ell_2^n)_{p_2}$.

To distinguish between $\iFSS(\cs_{\lore(w_1,p_1)}, \cs_{\lore(w_2,p_2)})$ and
$\iK(\cs_{\lore(w_1,p_1)}, \cs_{\lore(w_2,p_2)})$, note that
$\cs_{\lore(w,p)}$ contains complemented copies of $\ell_2$ and $\ell_p$.
Now recall (see \cite[Proposition~3.3]{SSTT}) that, for $u < v$, the formal
identity from $\ell_u$ to $\ell_v$ is finitely strictly singular, but not compact.
\end{proof}

\section{Non-commutative $L_p$ spaces: discrete case}\label{s:Cp}

In this section we study some operator ideals on the spaces $\cs_p$.
For this, we need some notation. Let $(e_k)_{k=1}^\infty$ be the canonical
basis in $\ell_2$. Let $P_n$ be the orthogonal projection from $\ell_2$ onto
$\span[e_1, \ldots, e_n]$, and $P_n^\perp = \one - P_n$. For convenience, set
$P_0 = 0$. If $\ce$ is a non-commutative symmetric sequence space, let $Q_n$ be
the projection on $\ce$, defined via $Q_n x = P_n x P_n$. Similarly,
let $R_n x = P_n^\perp x P_n^\perp$.

\subsection{Strictly singular, compact, and weakly compact operators}
In this subsection we study when the
operator ideals, introduced in Section \ref{s:intro}, coincide. We begin by establishing:

\begin{corollary}\label{c:main}
Suppose the Banach space $X$ satisfies one of two conditions:
\begin{enumerate}[(i)]
\item
$X = \cs_p$, with $1 \leq p < \infty$.
\item
$X = L_p(\tau)$, where $1 < p < \infty$, and $\tau$ is a normal
faithful finite trace on a hyperfinite von Neumann algebra.
\end{enumerate}
Then, for $T \in B(X)$, the following statements are equivalent:
\begin{enumerate}
\item
$T$ is not strictly singular.
\item
$T$ is not inessential.
\item
$X$ contains a subspace
$E$, isomorphic to either $\ell_2$ or $\ell_p$, so that $T|_E$ is an isomorphism, and
both $E$ and $T(E)$ are complemented in $X$.
\end{enumerate}
%
%
\end{corollary}

\begin{proof}
By \cite[Theorem 7.44]{Ai} any strictly singular or strictly cosingular operator
is inessential. Now suppose $T \in B(X)$ is not strictly singular.
By Theorems \ref{t:main} and \ref{thm:lor SS=IN}, there exists $E \subset X$
so that $E$ is isomorphic to either $\ell_2$ or $\ell_p$,
$T|_E$ is an isomorphism, and both $E$ and $T(E)$ are complemented.
A well-known description of inessential operators (see e.g.
\cite[Section 7.1]{Ai}, or \cite{P}) shows that $T$ is not inessential.
\end{proof}

\begin{remark}\label{r:subproj}
Alternatively, one could show that, for operators on $\cs_p$, the ideals of strictly
singular and inessential operators coincide by combining \cite[Theorem 7.51]{Ai}
with the subprojectivity of $\cs_p$, established in \cite{AL}.
\end{remark}

\begin{corollary}\label{c:duality}
Suppose $1 < p < \infty$, and $X$ is either $\cs_p$, or 
$L_p(\tau)$, where $\tau$ is a normal
faithful finite trace on a hyperfinite von Neumann algebra.
Then, for $T \in B(X)$, the conditions $(1)$, $(2)$, and $(3)$ from Corollary~\ref{c:main}
are equivalent to:
\begin{enumerate}
\setcounter{enumi}{3}
\item
$T^*$ is not strictly singular.
\item
$T$ is not strictly cosingular.
\end{enumerate}
\end{corollary}

\begin{proof}
$Z$ is reflexive, hence $T \in \iIN(Z)$ if and only if
$T^* \in \iIN(Z^*)$ (see \cite[Corollary 7.8]{Ai}).
An application of Corollary~\ref{c:main} completes the proof (4).
Duality between strict singularity and strict cosingularity
(see \cite[Theorem 7.53]{Ai}) yields $(4) \Leftrightarrow (5)$.
\end{proof}

\begin{remark}
In \cite{OS_Cstar} we show that the same result holds for $Z = \cs_\infty$.
Things are different for $Z = \cs_1$.
Combining \cite[Theorem~2.2]{Whi} and \cite{Fr}, we conclude that
$T^{*} \in \iSS( B(\ell_2))$ implies $T \in \iSS(\cs_1)$.
In contrast with Corollary \ref{c:duality}, the converse is false.
Indeed, let $j$ be a complemented embedding of $\ell_2$ into $\cs_1$,
$p$ a projection from $\cs_1$ onto the diagonally embedded $\ell_1$, and
$q$ a quotient map from $\ell_1$ to $\ell_2$. By Grothendieck's Theorem,
$q$ is $1$-summing, hence 
strictly singular.
Thus, $jqp$ is strictly singular. However, $(jqp)^* = p^* q^* j^*$ is not
strictly singular. Indeed, $q^*$ is an isometric embedding of $\ell_2$ into $\ell_\infty$,
$p^*$ is an isometric embedding of $\ell_\infty$ into $B(\ell_2)$, and
$j^* : B(\ell_2) \to \ell_2$ is a projection.

However, \cite{OS_Cstar} shows that any $T \in \iSS(\cs_1)_+$
is compact, hence $T^*$ is also compact, and, consequently, strictly singular.

The ideal of cosingular operators acting on $\cs_1$ sits properly
between the ideals of compact and strictly singular  operators.
Indeed,
strictly cosingular operators are inessential.
Thus,  Corollary~\ref{c:main} yields
$\iSCS(\cs_1) \subseteq \iSS(\cs_1)$. By \cite[Theorem 2.3.1]{AK},
there exists a surjective operator  $T:X \to Y$, where $X$ and $Y$
are complemented subspaces of $\cs_1$ isomorphic to $\ell_1$ and $\ell_2$, respectively.
Clearly, $T$ is a strictly singular operator. But, being surjective,
$T$ is not strictly cosingular. This implies that
$S=TP \in \iSS(\cs_1) \setminus \iSCS(\cs_1)$, where $P$ is a projection from
$\cs_1$ onto $X$. Finally, there is a strictly cosingular non-compact operator on $\cs_1$.
It is the canonical embedding of $\ell_1$ into $\ell_2$ combined with a projection on a copy of $\ell_1$. 
\end{remark}

The following result sharpens Proposition \ref{p:FSS not SS}
for $\cs_p$ spaces.

\begin{theorem}\label{t:C_p->C_q} The following holds:
\begin{enumerate}
\item $\iSS(\cs_p,\cs_q)=\iK(\cs_p,\cs_q)$,  if $\infty \ge  p\ge 2\ge q \ge 1$,
\item $\iSS(\cs_p,\cs_q) \supsetneq \iFSS(\cs_p,\cs_q) \supsetneq \iK(\cs_p,\cs_q)$
otherwise.
\end{enumerate}
\end{theorem}




\begin{proof}
Part $(2)$ follows from Proposition \ref{p:FSS not SS} (that proposition
does not cover the case of $p = \infty$, which can be established similarly).
To tackle $(1)$, suppose $\infty \ge  p\ge 2\ge q \ge 1$,
and $T:\cs_p \to \cs_q$ is not compact.
Then there is a weakly null sequence $(x_n)$
such that $Tx_n$ is bounded away from $0$.
Indeed, by the noncompactness of $T$, there exists a sequence $(z_n) \subset \ball(\cs_p)$,
so that $\inf_{n \neq m} \|Tz_n - Tz_m\| > 0$. By Rosenthal's characterization
of $\ell_1$ (see e.g. \cite[Theorem 10.2.1]{AK}), we can assume, by passing
to a subsequence if necessary, that $(z_n)$ is weakly Cauchy. Then the sequence
$x_n = (z_{2n} - z_{2n+1})/2$ has the desired properties.

First, consider $p \ne \infty$.
Clearly, $(x_n)$ contains a basic subsequence, thus, from \cite[Theorem 3.1]{AL} 
by passing to a subsequence,  $(x_n)$ can be considered equivalent to either
an $\ell_2$-basis, or an $\ell_p$-basis. Similarly, we can assume $(Tx_n)$ is equivalent
to either an $\ell_2$-basis, or an $\ell_q$-basis.
We recall Pitt's theorem  and the fact that $q<2$
to deduce that $T$ is an isomorphism on a copy of $\ell_2$. Hence the result follows.

To handle $p=\infty$, it suffices to note that, by the proof of  \cite[Theorem 3.1]{AL},
every basic sequence in $\cs_\infty$ contains either an $\ell_2$-basis, or a $c_0$-basis.
The rest of the argument proceeds similarly.
%
%
\end{proof}

In the case of operators on $\cs_2$ (which is isometric to a Hilbert space),
more can be said. Recall that a sequence space $\ce$ is said to have
\emph{lower $2$-estimate} if, for any finite sequence of disjoint elements $(x_i)$,
we have $\| \sum_i x_i \|^2 \geq C \sum_i \|x_i\|^2$, where the constant $C$
depends only on $\ce$.

\begin{proposition}\label{p:l2->low 2 est}
If a symmetric sequence space $\ce$ has lower $2$-estimate, then any
strictly singular operator from $\ell_2$ to $\se$ is compact.
\end{proposition}


\begin{proof}
Assume that there exists a norm one $T \in \iSS(\ell_2, \se) \backslash \iK(\ell_2, \se)$.
Find an orthonormal sequence $(h_k) \subset \ell_2$
so that, for any $k$, $\|T h_k\| > c > 0$. For any $n$, the range of $I - R_n$
is isomorphic to $\ell_2$, hence $(I - R_n) T$ is compact. In particular,
$\lim_k \|(I - R_n) T h_k\| = 0$. Passing to a subsequence if necessary,
we can assume the existence of $k_1 < k_2 < \ldots$ so that, for any $k$,
$\|(I - R_{n_k}) T h_k\| < 10^{-k} c$, and
$\|(I - Q_{n_{k+1}}) T h_k\| < 10^{-k} c$. Let
$H = \span[e_k : k \in \N]$, and
$x_k = (I - R_{n_k}) (I - Q_{n_{k+1}}) T h_k$. Then
$$
\|x_k - T h_k\| \leq \|I - R_{n_k}\| \|(I - Q_{n_{k+1}}) T h_k\| <
2 \cdot 10^{-k} ,
$$
and, therefore, the operator $S : H \to \se : h_k \mapsto Th_k - x_k$ is nuclear,
and, in particular, compact. Furthermore, we can write $x_k = u_k y_k v_k$,
where $u_k$ and $v_k$ are unitaries on the finite dimensional range of
the projection $P_{n_{k+1}} - P_{n_k}$, and $y_k$ is diagonal.
Let $j$ be the canonical diagonal embedding of $\ce$ into $\se$,
Now define the operator
$V : H \to \ce : e_k \mapsto j^{-1}y_k$.
By \cite[Proposition 2.1]{FHKT}, $V$ is compact. However,
$\|V h_k\| = \|y_k\| = \|x_k\| > c/2$ for any $k$, yielding
the desired contradiction.
\end{proof}



\subsection{Positive operators on Schatten spaces}\label{ss:+schat}
Next we turn to positive operators.
Below, we keep the notation introduced after the statement of
Proposition \ref{p:l2->low 2 est}. Recall two results from \cite{OS}.

\begin{lemma}\label{lem:diagonal}
Suppose $\ce$ is a separable symmetric sequence space, not containing $\ell_1$, and
$S : \se \to Z$ is compact ($Z$ is a Banach space). Then $\lim_n \|S|_{R_n(\se)}\| = 0$.
\end{lemma}

\begin{proposition} \label{prop:glid_hump}
Suppose $\ce$ is a separable symmetric sequence space.
Let $(A_n)$ be a weakly null positive sequence in $\se(H)$,
which contains no convergent subsequences. Then there exists $c > 0$
with the property that, for any $\vr \in (0,1)$, there exist sequences
$1 = n_1 < n_2 < \ldots$ and $0 = m_0 < m_1 < \ldots$, so that
$\inf_k \|A_{n_k}\| > c$, and
$$
\sum_k \|A_{n_k} - (P_{m_k} - P_{m_{k-1}}) A_{n_k} (P_{m_k} - P_{m_{k-1}})\| < \vr .
$$
In particular, there exists a subsequence
$(A_{n_k})$, equivalent to a disjoint positive sequence of
finite dimensional operators.
\end{proposition}


These results allow us to prove a ``non-commutative Pitt's Theorem''
for positive operators.

\begin{theorem}\label{t:C_p->C_q +}
For $1 \leq q < p < \infty$, $B(\cs_p,\cs_q)_{+}$=$\iK(\cs_p,\cs_q)_{+}$.
\end{theorem}

\begin{proof}
Suppose, for the sake of contradiction, that there exists a non-compact
$T \in B(\cs_p,\cs_q)_{+}$. By Lemma~\ref{lem:diagonal}
$ \inf_n \|TR_n\| >0$. Then there exists a sequence $(n_k)$, and
a normalized positive sequence  $(x_k)$ in $\cs_p$, so that
$x_k = (P_{n_k} - P_{n_{k-1}}) x_k (P_{n_k} - P_{n_{k-1}})$,
and $\|T x_k\| > c > 0$ for every $k$. By polarization, we can assume
that $x_k \geq 0$ for every $k$.
The sequence $(x_k)$ is equivalent to the standard basis of $\ell_p$, hence
weakly null. Therefore, the sequence $(Tx_k$) is weakly null as well.
Proposition~\ref{prop:glid_hump}
implies the existence of $k_1 < k_2 < \ldots$ so that the sequence $(Tx_{k_j})$
is equivalent to a standard basis of $\ell_q$. Thus, $T$ maps an $\ell_p$-basis
to an $\ell_q$-basis, which contradicts the boundedness of $T$.
\end{proof}

\begin{theorem}\label{t:C_p->C_p +}
For $1 \leq p < \infty$, and $T \in B(\cs_p)_{+}$, the following are equivalent:
\begin{enumerate}
\item
$T$ is compact.
\item
$T$ is strictly singular.
\item
There is no a subspace $E \subset \cs_p$, isomorphic to $\ell_p$,
so that $T|_E$ is an isomorphism, and both $E$ and $T(E)$ are complemented.
\end{enumerate}
\end{theorem}
\begin{proof}

The implications $(1) \Rightarrow (2) \Rightarrow (3)$ are trivial.
To establish $(3) \Rightarrow (1)$, it suffices to show that any
$T \in B(\cs_p)_{+} \backslash \iK(\cs_p)_{+}$ fixes a copy of $\ell_p$
(then proceed as in the proof of Proposition \ref{thm:lor SS=IN}).

Suppose first $1 < p < \infty$.
Proceeding as in the proof of Theorem \ref{t:C_p->C_q +}, we show that $T$ maps an
$\ell_p$-basis to an $\ell_p$-basis, hence $T$ is not strictly singular.

Let $p=1$.
 Then there exists a positive seminormalized sequence $(x_n)$ such that $Tx_n$ does not contain any convergent subsequences. Since $\cs_1$ is sequentially weakly complete, see \cite{Ak:67},  by passing to a subsequence, we may assume that $(x_n)$ is ether isomorphic to the  $\ell_1$-basis or 
it is weakly convergent. The later yields $(x_n)$ is norm convergent by \cite[Theorem~1.4.3]{OS}, which contradicts the way we have chosen $(x_n)$. Similar we obtain that $T(x_n)$ is equivalent to the $\ell_1$-basis. Therefore, $T$ is an isomorphism on a copy of $\ell_1$.
\end{proof}

\subsection{Dunford-Pettis operators}
In this subsection, we explore the compactness and singularity properties of
Dunford-Pettis operators. To introduce the relevant definitions,
suppose $X,Y$ are  Banach spaces.
$T$ is called \emph{$Z$-compact}
(\emph{$Z$-strictly singular}) if $T|_{Z^\prime}$ is compact (resp.~strictly singular)
whenever $Z^\prime \subset X$ is isomorphic to $Z$. Clearly, $Z$-compactness implies
$Z$-strict singularity. Furthermore, a Dunford-Pettis operator is $Z$-compact
for any reflexive $Z$.

In \cite{Ro76}, H.~Rosenthal proved that an operator $T \in B(L_1,Z)$ is
Dunford-Pettis if and only if it is $\ell_2$-strictly singular.
The previous paragraph shows that these two properties are equivalent to
$\ell_2$-compactness.
Below we establish a non-commutative version of Rosenthal's result.

\begin{proposition}\label{p:l_2 DP}
Suppose $Z$ is a Banach space. Then an operator $T \in B(\cs_1,Z)$ is Dunford-Pettis
if and only if it is $\ell_2$-compact.
\end{proposition}

\begin{proof} 
Since $\ell_2$ is a reflexive space, any Dunford-Pettis operator is $\ell_2$-compact.
Suppose $T$ is $\ell_2$-compact.
Then  \cite[Theorem~2.2 (i)]{Ar:80} implies that for any $n \in \N$ and $\vr > 0$
there exists $N = N(n,\vr)$ so that
$$
\|T|_{\span[E_{ij} : i \leq n, j > N]}\|,
\|T|_{\span[E_{ij} : i > N, j \leq n]}\| < \vr. 
$$
Now select a sequence $1 = u_1 < v_1 < u_2 < \ldots$, so that, for any $k$,
$\|T|_{X_k}\|, \|T|_{Y_k}\| < 4^{-k}$. For convenience, let $u_0 = v_0 = 0$.
Here,
$$
X_k = \span[E_{ij} : (i,j) \in A_k] , \, \,
Y_k = \span[E_{ij} : (i,j) \in B_k] ,
$$
$$  \eqalign{
&
A_k = \{(i,j) : i \leq u_k, j > v_k\} = [1,u_k] \times (v_k, \infty) ,
\cr
&
B_k = \{(i,j) : j \leq v_k, i > u_{k+1}\} = (u_{k+1},\infty) \times [1,v_k] .
}  $$
Note that the spaces $X_k$ and $Y_k$ are isomorphic to $\ell_2$, hence
$T|_{X_k}$ and $T_{Y_k}$ are compact. Therefore,
$T|_{\span[E_{ij} : (i,j) \in \cup_k A_k \times B_k]}$ is compact.

Moreover, $C = \N \times \N \backslash (\cup_k A_k \times B_k)$
is the disjoint union of the sets $C_k = [u_{k-1},u_{k+1}] \times [v_{k-1},v_k]$.
Then $\span[E_{ij} : (i,j) \in C] \subset \cs_1$ is isomorphic to
$X_0 = (\oplus_k \cs_1^{a_k,b_k})_{\ell_1}$, where
$a_k = u_{k+1} - u_{k-1} + 1$, and $b_k = v_k - v_{k-1} + 1$.
As $X_0$ is an $\ell_1$ sum of finite dimensional spaces, it has the
Schur property. Consequently, any operator on $X_0$ is Dunford-Pettis.

Let $P$ be the coordinate projection from $\cs_1$ onto
$\span[E_{ij} : (i,j) \in C]$, see \cite[Proposition~3]{AL}.
Note that $TP$ factors through $X_0$, while $T(1-P)$ factors
through $T|_{\span[E_{ij} : (i,j) \in \cup_k A_k \times B_k]}$.
Thus, both $TP$ and $T(1-P)$ are Dunford-Pettis. The same
property is inherited by $T = TP + T(1-P)$.
\end{proof}

Under certain conditions, every strictly singular operator is
Dunford-Pettis.

\begin{proposition}\label{p:SS->DP}
If symmetric sequence spaces $\ce$ and ${\mathcal{F}}$ does not contain $c_0$, and
$\iSS(\ce \oplus \ell_2, {\mathcal{F}} \oplus \ell_2) \subset
 \iDP(\ce \oplus \ell_2, {\mathcal{F}} \oplus \ell_2)$, then
$\iSS(\se, \cs_{\mathcal{F}}) \subset \iDP(\se, \cs_{\mathcal{F}})$.
\end{proposition}

\begin{proof}
Consider, for the sake of contradiction,
$T \in \iSS(\se, \cs_{\mathcal{F}}) \backslash \iDP(\se, \cs_{\mathcal{F}})$.
Then there exists a weakly null sequence $x_n \in \se$ so that
$\|Tx_n\| \geq c > 0$ for every $n$. By passing to a subsequence, and using
\cite{OS_SP}, 
we can assume that $(x_n)$ ($(Tx_n)$) is contained
in a complemented subspace $Z_\ce$ ($Z_{\mathcal{F}}$) of $\se$ ($\cs_{\mathcal{F}})$),
respectively, isomorphic to either (resp. $\ce$, $\ell_2$, or $\ce \oplus \ell_2$
(resp. ${\mathcal{F}}$, $\ell_2$, or ${\mathcal{F}} \oplus \ell_2$).
Let $J$ be the embedding of $Z_1$ into $\ce$, and $P$ be a projection onto $Z_2$.
Then $P T J$ is strictly singular, hence Dunford-Pettis, contradicting
$\limsup_n \|P T J x_n\| = \limsup_n \|T x_n\| \geq c$.
\end{proof}

\begin{corollary}\label{c:C1->low2est DP}
Suppose a symmetric sequence space ${\mathcal{F}}$ satisfies the lower $2$-estimate.
Then $\iSS(\cs_1, \cs_{\mathcal{F}}) \subset \iDP(\cs_1, \cs_{\mathcal{F}})$.
\end{corollary}

\begin{proof}
By \cite[Section 1.f]{LT2}, ${\mathcal{F}}$ has non-trivial cotype, hence it
cannot contain $c_0$. To finish the proof, we need to establish
the inclusion
\begin{equation}
\iSS(\ell_1 \oplus \ell_2, {\mathcal{F}} \oplus \ell_2) \subset
 \iDP(\ce \oplus \ell_2, {\mathcal{F}} \oplus \ell_2) ,
\label{eq:inclusion}
\end{equation}
and then apply Proposition \ref{p:SS->DP}.
To establish \eqref{eq:inclusion}, recall that $\ell_1$ has the Schur Property,
hence any operator from $\ell_1$ is Dunford-Pettis. By Proposition~\ref{p:l2->low 2 est},
any strictly singular operator from $\ell_2$ to ${\mathcal{F}} \oplus \ell_2$
is compact, hence Dunford-Pettis.
\end{proof}

\section{Products of strictly singular operators}\label{s:prod}

We apply the results of this paper to generalize a theorem from \cite{Mil}
to a non-commutative setting.

\begin{theorem}\label{t:milman}
Suppose a Banach space $Z$ is isomorphic to either $\cs_p(H)$ ($1 \leq p \leq \infty$,
$H$ is a Hilbert space),
or to $L_p(\tau)$ ($1 < p < \infty$, $\tau$ is a normal faithful finite trace
on a semifinite von Neumann algebra). If $T, S \in \iSS(Z)$, then $TS$ is compact.
\end{theorem}

\begin{proof}

Suppose, for the sake of contradiction, that $TS$ is not compact.
Then there exist $c > 0$, and a sequence $(x_i) \subset \ball(Z)$, so that
$\inf_i \dist(TSx_i, \span[TSx_j : 1 \leq j < i]) > c$.

First consider the case of $Z = L_p(\tau)$,
with $2 \leq p < \infty$. We can assume that $(x_i)$ is
weakly null. Indeed, by passing to a subsequence, we guarantee that
$(x_i)$ is weakly Cauchy. Now the sequence
$x_i^\prime = (x_{2i} - x_{2i-1})/2$ lives in $\ball(Z)$, and
is weakly null. An easy calculation shows that, for any $i$,
$$
\dist(TSx_i^\prime, \span[TS x_j^\prime : 1 \leq j \leq i-1]) \geq
\frac{1}{2} \dist(TSx_i, \span[TS x_j : 1 \leq j \leq 2i-1]) \geq
\frac{c}{2} .
$$


Non-commutative Kadec-Pelczynski dichotomy \cite{PX} shows that any unconditional
seminormalized sequence contains a subsequence equivalent to either
the $\ell_2$-basis, or $\ell_p$ basis. If $(x_i)$ is equivalent to
the $\ell_p$-basis, then $(Sx_i)$ must have a subsequence equivalent
to the $\ell_p$-basis as well, contradicting the strict singularity
of $S$. If $(x_i)$ is equivalent to the $\ell_2$-basis, then $(Sx_i)$
must have a subsequence equivalent to the $\ell_p$-basis, and the
same holds for $(TSx_i)$, which is impossible, due to the strict
singularity of $T$.

The case of $Z = \cs_p(H)$ ($1 < p < \infty$) is handled similarly.
As above, we can assume that the sequence $(x_i)$ is weakly null.
By passing to a further subsequence, we can assume that the sequences
listed above are basic. Passing to a subsequence once again, and invoking
\cite[Theorem 1]{AL}, we can assume that each of the sequences $(x_i)$,
$(Sx_i)$, and $(TSx_i)$ is equivalent either to the $\ell_p$-basis,
or to the $\ell_2$-basis.

Now we turn to $Z = \cs_1(H)$. Applying Rosenthal's $\ell_1$ theorem
(see e.g. \cite[Theorem 10.2.1]{AK}) and passing to a subsequence,
we can assume that, for each of the sequences $(x_i)$, $(Sx_i)$, and
$(TSx_i)$, one of the following is true: (i) the sequence is weakly Cauchy, or
(ii) the sequence is equivalent to the $\ell_1$-basis.

Suppose first $(x_i)$ is weakly Cauchy. Then, as above, we can assume
that $(x_i)$ is weakly null. Consequently, the sequences $(Sx_i)$ and
$(TSx_i)$ are weakly null as well. Passing to a further subsequence,
we can assume that all three sequences are basic as well. By passing
to a subsequence, and invoking \cite[Theorem 1]{AL}, we can assume that
all three sequences listed above are equivalent to the $\ell_2$-basis,
which contradicts the strict singularity of $T$ and $S$.

Now suppose $(x_i)$ is equivalent to the $\ell_1$ basis. Due to the strict
singularity of $S$, $(Sx_i)$ cannot be equivalent to the $\ell_1$-basis,
hence is is weakly Cauchy. Moreover, we can assume that $(Sx_i)$ is weakly
null. Indeed, we can pass to the sequence $x_i^\prime = (x_{2i} - x_{2i-1})/2$,
which is equivalent to the $\ell_1$-basis. Then $(TSx_i)$ is weakly null as well.
Passing to a subsequence as before, we obtain that both $(Sx_i)$ and $(TSx_i)$
are equivalent to the $\ell_2$-basis, contradicting the strict singularity of $T$.

Finally, suppose $Z = \cs_\infty(H)$, or $Z = L_p(\tau)$ with $1 < p < 2$. 
By Corollary \ref{c:duality},
$S^*$ and $T^*$ are strictly singular as well. By the reasoning above,
$S^* T^* = (TS)^*$ is compact, hence so is $TS$.
\end{proof}

\end{document}